\documentclass{amsart}

\usepackage[T1]{fontenc}
\usepackage{lmodern}
\usepackage{amsmath,amssymb,amsthm}
\usepackage{mathtools}
\usepackage[colorlinks=true,linkcolor=blue,citecolor=blue,urlcolor=blue]{hyperref}

\newtheorem{theorem}{Theorem}
\newtheorem{conjecture}{Conjecture}
\newtheorem{lemma}[theorem]{Lemma}
\theoremstyle{definition}

\theoremstyle{remark}

\newcommand{\T}{\mathcal T}
\newcommand{\I}{\mathcal I}
\newcommand{\F}{\mathcal F}
\newcommand{\N}{\mathbb N}
\newcommand{\wt}{\operatorname{wt}}

\title{AN ERD\H{O}S--KO--RADO THEOREM FOR TILINGS}
\author{CASEY TOMPKINS}
\email{casey.tompkins@renyi.hu}
\date{}

\begin{document}

\begin{abstract}
We prove an Erd\H{o}s--Ko--Rado type extremal result for tilings of a $1 \times n$ chessboard by tiles whose lengths belong to a set $\Lambda$.
Two tilings are said to intersect if they contain a tile spanning the same set of squares.
We prove that if $1\in\Lambda$, then the maximum size of an intersecting family of tilings is attained by the set of all tilings containing a fixed singleton tile at one of its ends.
This result generalizes a theorem of Butler, Horn and Tressler, which is equivalent to the case $\Lambda=\{1,2\}$.
\end{abstract}

\maketitle
\vspace{-2em}
\section{Introduction and Notation}

The classical theorem of Erd\H{o}s, Ko and Rado~\cite{EKR1961} determines the maximum size of a pairwise intersecting family of $k$-element subsets of an $n$-element set in the nontrivial range $n\geq 2k$.
The same paper also gives the sharp bound for an arbitrary intersecting family of subsets of an $n$-element set.
Following this initial work, analogues have been proved in many other settings, including finite vector spaces~\cite{Hsieh1975}, multisets~\cite{Anderson1988}, set partitions~\cite{MeagherMoura2005, KuRenshaw2008,DyckMeagher2015}, integer partitions~\cite{Borg2014,Borg2016}, permutations~\cite{DezaFrankl1977}, graph families~\cite{SimonovitsSos1978} and independent sets in graphs~\cite{HolroydTalbot2005}.

Butler, Horn and Tressler~\cite{ButlerHornTressler2010} introduced a notion of intersection for families of domino tilings of chessboards of fixed dimensions.
They say that two domino tilings intersect if they both contain a domino spanning the same two squares of the chessboard, and they investigate the maximum size of an intersecting family of domino tilings.
For a $2\times n$ chessboard, they prove that the family of tilings containing one fixed vertical domino in an end column is optimal.
Since the total number of domino tilings of a $2\times n$ chessboard is the $(n+1)$st Fibonacci number $F_{n+1}$, their result gives the sharp upper bound $F_n$.
They also consider the analogous problem for $3\times 2n$ chessboards and prove a sharp bound on the maximum size of an intersecting family of tilings.

The $2\times n$ chessboard result can be reformulated in one dimension. A vertical domino in a column corresponds to a tile of length $1$ on a $1\times n$ board, while a pair of horizontal dominoes occupying two adjacent columns corresponds to a tile of length $2$. In this setting, we say two tilings intersect if they share a tile spanning the same positions. The result of Butler, Horn and Tressler~\cite{ButlerHornTressler2010} implies that an intersecting family of such tilings has size at most $F_n$.

It is natural to consider this problem for general tile sets, and we now make this precise. Let $\Lambda\subseteq\N$ be a set of allowed tile lengths.
For a positive integer $n$ let $[n]=\{1,2,\ldots,n\}$.
A $\Lambda$-tile in $[n]$ is an interval of the form $[i+1,i+\lambda]$, where $0\leq i\leq n-\lambda$ and $\lambda\in\Lambda$.
A $\Lambda$-tiling of $[n]$ is a partition of $[n]$ into $\Lambda$-tiles.
Equivalently, it is a word
\[
    T=(x_1,x_2,\ldots,x_k),
\]
with $x_i\in\Lambda$ and $x_1+x_2+\cdots+x_k=n$. The entry $x_i$ corresponds to the interval $[p_{i-1}+1,p_i]$, where $p_i=x_1+\cdots+x_i$ and $p_0=0$.

Let $\T_n^\Lambda$ be the set of all $\Lambda$-tilings of $[n]$.
For a tiling $T$, write $\I(T)$ for the set of tiles in $T$.
We say two tilings $T,U\in\T_n^\Lambda$ intersect if
\[
    \I(T)\cap\I(U)\ne\varnothing.
\]
A family $\F\subseteq\T_n^\Lambda$ is intersecting if every two distinct tilings in $\F$ intersect.

Let
\[
    a_m^\Lambda \coloneqq |\T_m^\Lambda|.
\]
We use the conventions $a_0^\Lambda=1$ and $a_m^\Lambda=0$ for $m<0$.
Then, for $m\geq 1$,
\[
    a_m^\Lambda=\sum_{\lambda\in\Lambda} a_{m-\lambda}^\Lambda.
\]
If $1\in\Lambda$, the family
\[
    A_n=\{T\in\T_n^\Lambda:[1,1]\in\I(T)\}
\]
has size $a_{n-1}^\Lambda$, and the same is true for the family of all tilings containing the singleton tile $[n,n]$.
We call these two families \emph{endpoint stars}. Our main result is the following.

\begin{theorem}\label{thm:main}
Let $n\geq 1$, let $\Lambda\subseteq\N$ with $1\in\Lambda$ and let $\F\subseteq\T_n^\Lambda$ be an intersecting family.
Then
\[
    |\F|\leq a_{n-1}^\Lambda.
\]
\end{theorem}

When $\Lambda=\{1,2\}$, the numbers $a_n^\Lambda$ are the Fibonacci numbers $F_{n+1}$.
Therefore Theorem~\ref{thm:main} gives $|\F|\leq F_n$, which is equivalent to the $2\times n$ tilings theorem of Butler, Horn and Tressler~\cite{ButlerHornTressler2010}.

Equality in Theorem~\ref{thm:main} is attained by the endpoint stars. We remark that other families also attain equality. For example, when $\Lambda=\{1,2\}$ and $n$ is even, one can start with the $[1,1]$ endpoint star and replace $(1,1,2^{\frac{n-2}{2}})$ by $(2,1^{n-2})$, yielding another intersecting family of the same size. Here and elsewhere we use exponents to denote repeated copies of an entry in the word, and sometimes we simply use a concatenated list of symbols in place of the vector notation.

We conjecture that an analogue of Theorem~\ref{thm:main} for arbitrary tile sets also holds if $n$ is sufficiently large.

\begin{conjecture}\label{conj:main}
Let $\Lambda\subseteq\N$ be finite and nonempty and let $n$ be sufficiently large. Set $\lambda_0 = \min \Lambda$ and let $\F\subseteq\T_n^\Lambda$ be an intersecting family.
Then
\[
    |\F|\leq a_{n-\lambda_0}^\Lambda.
\]
\end{conjecture}

The proof by Butler, Horn and Tressler~\cite{ButlerHornTressler2010} (in our terminology) involved constructing in a block-wise fashion an injective map $\Phi$ from the set of tilings beginning with $2$ to the set of tilings beginning with $1$, such that, for each $T$, the tilings $T$ and $\Phi(T)$ are disjoint. Once such a map is constructed, their desired bound is immediate. 

The proof of Theorem~\ref{thm:main} proceeds by applying a compression argument which reduces the problem to considering intersections of the singleton tiles. This compression argument is similar to one used by Ku and Renshaw~\cite{KuRenshaw2008} for intersecting families of set partitions. We then construct a suitable injective map into the set of tilings beginning with a singleton tile to obtain the required upper bound. The proof of Theorem~\ref{thm:main} is given in Section~\ref{mainproof}.

\section{Proof of Theorem~\ref{thm:main}} \label{mainproof}
Fix a set of allowed tile lengths $\Lambda$ containing $1$. Let $I\subseteq[n]$ be an interval of length $d\in\Lambda$, where $d\geq 2$.
For a tiling $T$ with $I\in\I(T)$, write $T'$ for the tiling obtained from $T$ by replacing the single tile $I$ by~$d$ singleton tiles.
For a family $\F\subseteq\T_n^\Lambda$ and $T\in\F$, define
\[
    c_{\F,I}(T)=
    \begin{cases}
    T', & \text{if } I\in\I(T)\text{ and }T'\notin\F,\\
    T, & \text{otherwise,}
    \end{cases}
\]
and let
\[
    C_I(\F)=\{c_{\F,I}(T):T\in\F\}.
\]

\begin{lemma}\label{lem:compression}
If $\F\subseteq\T_n^\Lambda$ is intersecting, then $|C_I(\F)|=|\F|$ and $C_I(\F)$ is intersecting.
\end{lemma}

\begin{proof}
To prove $|C_I(\F)|=|\F|$, it is sufficient to show $c_{\F,I}:\F\to C_I(\F)$ is injective. 
Suppose $c_{\F,I}(T) = c_{\F,I}(U)$ for $T,U\in \F$.  If both $T$ and $U$ are left fixed by $c_{\F,I}$, then clearly $T=U$.  If both $T$ and $U$ are changed by $c_{\F,I}$, we have $T'=U'$ and so $T=U$.  It cannot be that exactly one of $T$ and $U$ is changed by $c_{\F,I}$ since the image of an unchanged tiling is in $\F$ and the image of a changed tiling is outside~$\F$. Thus $c_{\F,I}$ is injective.

Suppose now that $X,Y\in C_I(\F)$ are disjoint.
If $X,Y\notin\F$, then $X=A'$ and $Y=B'$ for some $A,B\in\F$ containing $I$, so $X$ and $Y$ share the singleton subtiles of $I$, a contradiction.
If $X,Y\in\F$, this contradicts that $\F$ is intersecting.
Thus, without loss of generality, $X=A'\notin\F$ for some $A\in\F$ containing $I$, while $Y\in\F$.

Since $A$ and $Y$ intersect but $A'=X$ is disjoint from $Y$, their only possible common tile is $I$.
Thus $\{I\}=\I(Y) \cap \I(A)$.
As $Y\in C_I(\F)\cap\F$ and $I\in\I(Y)$, we must have $Y'\in\F$.
Now $A$ and $Y'$ intersect since they are both in $\F$.  This is a contradiction since the unique tile common to $A$ and $Y$ is $I$ but $Y'$ contains singleton tiles in place of $I$.
\end{proof}

Starting from an intersecting family $\mathcal G_0$, repeatedly apply compressions of the above form, over intervals $I$ of length $d\in\Lambda$ with $d\geq 2$, whenever they change the current family. At each step the size of the family is unchanged, while $\sum_{T\in\mathcal G}|\I(T)|$, where $\mathcal G$ denotes the current family, strictly increases and is bounded above by $n|\mathcal G_0|$. Hence the process terminates.

By Lemma~\ref{lem:compression}, applying this process to an intersecting family $\F$ gives an intersecting family $\F^*$ such that $|\F^*|=|\F|$. Moreover, $\F^*$ has the following property: if $T\in\F^*$ contains a non-singleton tile $I$ of length $d$, and $T'$ is obtained from $T$ by replacing $I$ by its $d$ singleton subtiles, then $T'\in\F^*$. 

For a tiling $T$, define its \emph{singleton set} by
\[
    \sigma(T)=\{i\in[n]:[i,i]\in\I(T)\}.
\]

\begin{lemma}\label{lem:singleton-intersecting}
For all distinct $T,U\in\F^*$, we have
\[
    \sigma(T)\cap\sigma(U)\ne\varnothing.
\]
\end{lemma}

\begin{proof}
Suppose instead that $T,U\in\F^*$ are distinct and $\sigma(T)\cap\sigma(U)=\varnothing$.
Since $\F^*$ is intersecting, $T$ and $U$ share at least one tile, and no shared tile can be a singleton.
Refine $T$ by splitting every non-singleton tile that is also a tile of $U$ into singleton tiles and call the resulting tiling $T'$. Then, by iterating the property of $\F^*$ discussed above, $T'\in\F^*$.
But $T'$ and $U$ are disjoint, contradicting that $\F^*$ is intersecting.
\end{proof}

It remains to prove the upper bound for families whose singleton sets are pairwise intersecting.

Let
\[
    H=\Lambda\setminus\{1\}.
\]
If $H=\varnothing$, then there is only one tiling of $[n]$, the all-singleton tiling, and Theorem~\ref{thm:main} is immediate.
Hence we assume $H\ne\varnothing$, and let $d=\min H$.

We refer to an element of $H$ as a heavy length.
A heavy word is a finite (possibly empty) sequence $W=(w_1,\ldots,w_p)$ all of whose entries belong to $H$.
Define the weight of $W$ to be
\[
    \wt(W)=w_1+\cdots+w_p,
\]
where we set the weight of the empty word to be $0$.
For $m\geq 0$, let
\[
    R_m=\{W: W\text{ is a heavy word and }\wt(W)=m\}.
\]
Let $r_m=|R_m|$. Then, $r_0=1$, and $r_m=0$ for $1\leq m<d$.
For $m\geq 1$,
\[
    r_m=\sum_{h\in H}r_{m-h},
\]
with $r_j=0$ for $j<0$.

\begin{lemma}\label{lem:sliding}
For all integers $a\geq 1$ and $m\geq 1$,
\[
    \sum_{j=a}^{a+m-1}r_j
    \geq
    \sum_{j=1}^{m}r_j.
\]
\end{lemma}

\begin{proof}
Since $d\in H$, appending a heavy length to a word gives an injection from $R_j$ into $R_{j+d}$ for each $j\geq 0$,
and thus $r_{j+d}\geq r_j$.
Write $a=qd+b$ where $1\leq b \leq d$ and $q\geq 0$.
Then
\[
    \sum_{j=a}^{a+m-1}r_j \geq \sum_{j=b}^{b+m-1}r_j.
\]
Also
\[
    \sum_{j=b}^{b+m-1}r_j  +\sum_{j=1}^{b-1}r_j
    =
    \sum_{j=m+1}^{b+m-1}r_j  + \sum_{j=1}^{m}r_j.
\]
Since $b\leq d$, we have $r_j=0$ for all $1\leq j\leq b-1$, and the result follows.
\end{proof}

Our next aim is to construct an injection mapping each heavy word followed by a positive number of $1$'s to a word of the same weight beginning with $1$, such that the image has a disjoint set of singleton tiles. To do so, we will define a graph and apply Hall's theorem.

Fix $L\geq 2$, and define
\[
    X_L=\{(M,W):2\leq M\leq L-1,\ W\in R_M\}
\]
and
\[
    Y_L=\{(s,V):1\leq s\leq L-1,\ V\in R_{L-s}\}.
\]
We need an injection $\psi_L:X_L \to Y_L$
such that, whenever $\psi_L(M,W)=(s,V)$, we have $s\leq M$.
Then replacing $W1^{L-M}$ by $1^sV$ and applying the condition $s\leq M$ ensures that the new singleton run lies inside the original heavy block.

\begin{lemma}\label{lem:internal-code}
For every $L\geq 2$, such an injection $\psi_L$ exists.
\end{lemma}

\begin{proof}
Consider a bipartite graph with parts $X_L$ and $Y_L$.
Join $(M,W)$ to $(s,V)$ exactly when $s\leq M$.
The neighborhood of a vertex of $X_L$ depends only on $M$, and these neighborhoods are nested as $M$ increases.
Thus Hall's condition reduces to checking that, for every $2\leq m\leq L-1$,
\[
    \sum_{M=2}^{m}r_M
    \leq
    \sum_{s=1}^{m}r_{L-s}.
\]
Since $r_1=0$, the left-hand side is $\sum_{M=1}^{m}r_M$.
The right-hand side is
\[
    \sum_{s=1}^{m}r_{L-s}
    =
    \sum_{j=L-m}^{L-1}r_j.
\]
By Lemma~\ref{lem:sliding}, with $a=L-m$,
\[
    \sum_{j=L-m}^{L-1}r_j
    \geq
    \sum_{j=1}^{m}r_j.
\]
Hall's condition holds, and so the desired injection exists.
\end{proof}

Fix one such injection $\psi_L$ for each $L\geq 2$. If a tiling ends with a heavy block, then we need to use a different encoding for this block. For $M\geq 2$, define
\[
    E_M=\{1^M\}\cup
    \{1^aV1: a\geq 1,\ V\in R_h\text{ for some }h\geq 2,\ a+h+1=M\}.
\]
Every word in $E_M$ has total weight $M$, begins with $1$ and ends with $1$.
Moreover, every word in $E_M$ other than $1^M$ has terminal run of $1$'s of length exactly $1$.

\begin{lemma}\label{lem:final-code}
For every $M\geq 2$, there is an injection from $R_M$ into $E_M$.
\end{lemma}

\begin{proof}
We have
\[
    |E_M|=1+\sum_{j=2}^{M-2}r_j.
\]
On the other hand,
\[
    r_M=\sum_{h\in H}r_{M-h}.
\]
In this sum, the term $r_0=1$, if it appears, is covered by the initial $1$ in the formula for $|E_M|$.
The term $r_1$ is zero.
Every other nonzero term has index between $2$ and $M-2$, and is included in $\sum_{j=2}^{M-2}r_j$.
Thus $|E_M|\geq r_M$, so an injection from $R_M$ into $E_M$ exists.
\end{proof}

Fix one such injection $\eta_M$ for each $M\geq 2$. Let
\[
    B_n=\T_n^\Lambda\setminus A_n.
\]
We construct an injection $\Phi:B_n \to A_n$ such that $\sigma(T)\cap\sigma(\Phi(T))=\varnothing$
 for every $T\in B_n$.

Take $T\in B_n$.
Since $T$ does not begin with a singleton, it has a unique decomposition
\[
    T=W_1 1^{t_1} W_2 1^{t_2}\cdots W_q 1^{t_q},
\]
where each $W_i$ is a nonempty maximal consecutive block of heavy entries, $t_i\geq 1$ for $i<q$ and $t_q\geq 0$.
For every $i$ with $t_i\geq 1$, set
\[
    M_i=\wt(W_i),
    \qquad
    L_i=M_i+t_i.
\]
Since $W_i$ is nonempty and heavy, $M_i\geq 2$, and since $t_i\geq 1$, we have $2\leq M_i\leq L_i-1$.
Thus $(M_i,W_i)\in X_{L_i}$.
Write
\[
    \psi_{L_i}(M_i,W_i)=(s_i,V_i),
\]
and replace the block $W_i1^{t_i}$ by $1^{s_i}V_i$.

If $t_q=0$, then $W_q$ is a final heavy block.
In this case replace $W_q$ by
\[
    \eta_{\wt(W_q)}(W_q).
\]
Concatenating the resulting words gives a tiling of $[n]$.
Define the resulting tiling to be $\Phi(T)$.
The first replacement word begins with $1$, so $\Phi(T)\in A_n$.

\begin{lemma}\label{lem:disjoint-singletons}
For every $T\in B_n$,
\[
    \sigma(T)\cap\sigma(\Phi(T))=\varnothing.
\]
\end{lemma}

\begin{proof}
Consider an internal block $W_i1^{t_i}$ with $t_i\geq 1$.
Within this block, the singleton tiles of $T$ occur in the final $t_i$ relative positions
\[
    M_i+1,M_i+2,\ldots,L_i.
\]
The image block is $1^{s_i}V_i$, whose singleton tiles occur in the first $s_i$ relative positions
\[
    1,2,\ldots,s_i.
\]
By construction of $\psi_{L_i}$, we have $s_i\leq M_i$, so these two sets of relative positions are disjoint.
If $t_q=0$, the final block $W_q$ contains no singleton tiles in $T$, so any singleton tiles introduced by $\eta_{\wt(W_q)}(W_q)$ cannot coincide with singleton tiles of $T$.
The starting and ending positions of each block are preserved, and thus $\sigma(T)\cap\sigma(\Phi(T))=\varnothing$.
\end{proof}

\begin{lemma}\label{lem:Phi-injective}
The map $\Phi:B_n\to A_n$ is injective.
\end{lemma}

\begin{proof}
It is enough to show that for each $Y$ in the image of $\Phi$, the preimage of $Y$ can be determined uniquely.
Suppose $Y$ ends with a heavy length, and take the unique decomposition
\[
    Y=1^{s_1}V_1 1^{s_2}V_2\cdots 1^{s_q}V_q,
\]
where the $V_i$ are nonempty maximal heavy words. Let
$L_i=s_i+\wt(V_i)$. Since $\psi_{L_i}$ is injective, $(s_i,V_i)$ has a unique preimage
$(M_i,W_i)$, and the original block is $W_i1^{L_i-M_i}$.

Now suppose that $Y$ ends with $1$. Then the last heavy block of the
preimage was replaced by its image under one of the maps $\eta_M$.
Let $\tau$ be the length of the terminal run of $1$'s in $Y$.
If $\tau\geq 2$, this final $\eta$-image must be $1^\tau$, since every
other word in every $E_M$ has terminal run of $1$'s of length exactly $1$. Thus, by injectivity $\eta_\tau$, the final heavy block is recovered uniquely.
If $\tau=1$, this final $\eta$-image is of the form $1^aV1$, where $V$ is the maximal heavy block immediately before the last
$1$, and $1^a$ is the maximal singleton run immediately before $V$.
Then $M=a+\wt(V)+1$, and injectivity of $\eta_M$ recovers the final
heavy block.

After removing this final $\eta$-image, the remaining prefix is either
empty or ends with a heavy length, so it is decoded by the first paragraph.
Hence $Y$ determines $T$ uniquely, and $\Phi$ is injective.
\end{proof}

\begin{proof}[Proof of Theorem~\ref{thm:main}]
Let $\F\subseteq\T_n^\Lambda$ be intersecting.
If $\Lambda=\{1\}$, then there is only one tiling and the result is immediate, so assume $H\ne\varnothing$.
By Lemma~\ref{lem:compression} and the argument in the paragraph following it, there exists an intersecting family $\F^*$ such that $|\F^*|=|\F|$,
and such that replacing any non-singleton tile in a tiling of $\F^*$ with singleton tiles also yields a tiling in $\F^*$.
By Lemma~\ref{lem:singleton-intersecting}, the singleton sets in $\F^*$ are pairwise intersecting.

The map $\Phi:B_n\to A_n$ is injective by Lemma~\ref{lem:Phi-injective}, and by Lemma~\ref{lem:disjoint-singletons} it pairs every $T\in B_n$ with a tiling $\Phi(T)\in A_n$ whose singleton set is disjoint from $\sigma(T)$.
Hence $\F^*$ cannot contain both members of any pair
\[
    \{T,\Phi(T)\},
    \qquad T\in B_n.
\]
These pairs are disjoint, and every unpaired tiling lies in $A_n$.
Therefore
\[
    |\F|=|\F^*|
    \leq |A_n|
    =a_{n-1}^\Lambda,
\]
and the proof of Theorem~\ref{thm:main} is complete.
\end{proof}

\section{Acknowledgements}
We thank the participants in the 2019 Eml\'ekt\'abla workshop in Hungary, where the author proposed this problem, for several discussions on the topic. We also acknowledge the use of ChatGPT 5.5 as a writing and brainstorming aid in developing parts of the proof, in particular the formulation needed to apply Hall’s theorem.

\end{document}